\documentclass[11pt,leqno,english]{smfart}

\usepackage{dsfont,eucal,mathptmx,newtxtext,newtxmath}

\usepackage[text={6.5in,9.2in},centering]{geometry}

\usepackage[dvipsnames]{xcolor}   
\usepackage{xparse}
\usepackage{xr-hyper}
\definecolor{brightmaroon}{rgb}{0.76, 0.13, 0.28}
\usepackage[linktocpage=true,colorlinks=true,hyperindex,citecolor=teal,linkcolor=blue]{hyperref}   

\renewcommand{\labelenumi}{\upshape(\arabic{enumi})}

\newcommand{\spec}{\mathrm{Spec}(S)}
\newcommand{\s}{S}
\newcommand{\id}{\mathrm{Id}(S)}
\newcommand{\idp}{\mathrm{Id}^+(S)}

\newcommand{\mx}{\mathrm{Max}(S)}
\DeclareMathOperator{\jr}{\mathcal{J}(S)}
\DeclareMathOperator{\z}{\mathrm{Id}_z(S)}
\newcommand{\ma}{\mathfrak{a}}
\newcommand{\mb}{\mathfrak{b}}
\newcommand{\mc}{\mathfrak{c}}
\newcommand{\p}{\mathfrak{p}}
\newcommand{\q}{\mathfrak{q}}
\newcommand{\m}{\mathfrak{m}}
\newcommand{\mz}{\mathfrak{z}}
\DeclareMathOperator{\cz}{\mathrm{c}\;\!\!\ell_{\!z}}

\newtheorem{theorem}{Theorem}[section]
\newtheorem{proposition}[theorem]{Proposition}
\newtheorem{lemma}[theorem]{Lemma}
\newtheorem{corollary}[theorem]{Corollary}
\theoremstyle{definition}
\newtheorem{definition}[theorem]{Definition}
\newtheorem{example}[theorem]{Examples}
\newtheorem{remark}[theorem]{Remark}

\numberwithin{equation}{section}

\renewcommand{\labelenumi}{\upshape(\arabic{enumi})}

\begin{document}

\author{Amartya Goswami}
\address{
[1] Department of Mathematics and Applied Mathematics, University of Johannesburg, 2006, South Africa. [2] National Institute for Theoretical and Computational Sciences (NITheCS), Johannesburg, 2006, South Africa.}
\email{agoswami@uj.ac.za}

\title{On $z$-ideals and $z$-closure operations of semirings, I}








\maketitle

\begin{abstract}
The aim of this series of papers is to study $z$-ideals of semirings. In this article,  we introduce some distinguished classes of $z$-ideals of semirings, which include $z$-prime, $z$-semiprime,  $z$-irreducible, and $z$-strongly irreducible ideals and study some of their properties. Using a $z$-closure operator, we show the equivalence of these classes of ideals with the corresponding $z$-ideals that are prime, semirprime, irreducible, and strongly irreducible respectively.  
\end{abstract}
  
\section{Introduction}

Since the introduction of $z$-ideals in the context of rings of continuous functions  in \cite{Koh57}, a significant amount of interest has been generated in studying these types of ideals. The textbook  \cite{GJ60} became an authoritative source for $z$-ideals in these rings. Some further references in this direction include \cite{AKRA99, AM07, AP20, Mar72, HMW03, Don80, Mas80, Man68}. The definition of a $z$-ideal in a commutative ring first appeared in \cite{Mas73} and has since been expanded in various directions. For example, the study of $z$-ideals and $z^0$-ideals in power series rings can be found in \cite{AM13, MB20}. Higher-order $z$-ideals and extensions of $z$-ideals in rings have been discussed in \cite{DO16, ABN20, BM20, MB22}.  The concept of relative $z$-ideals was introduced in \cite{AAT13}. An extension of $z$-ideals for noncommutative rings can be found in \cite{MJ20}. For a study of fuzzy $z$-ideals, we refer to \cite{XZ09}. 

From a functional analysis perspective, $z$-ideals have been studied in \cite{HP80-1, HP80, Pag81} for Riesz spaces. In a more abstract setting, $z$-ideals have been introduced for lattices and multiplicative lattices, respectively, in \cite{JK19} and \cite{MC19}. In the context of point-free topology, these ideals have been extensively investigated in \cite{Dub16, Dub18}.

The objective of this series of articles is to explore $z$-ideals of commutative semirings. A semiring serves as a generalization of a ring by relaxing the requirement for additive inverses.  These algebraic structures have drawn a considerable interest due to the development of tropical geometry, where one studies an algebraic variety  through its combinatorial shadow, which takes the form of a polyhedral complex derived from the underlying set of the algebraic variety and a valuation on a ground field (see \cite{CC19, GG16, JK19, Lor19, Vir10}). As a natural extension of rings of continuous functions, \cite{VMS19} explored $z$-ideals in semirings of continuous functions. Additionally, the concept of $z$-ideals has been introduced for positive semirings in \cite{Moh14}.

In this paper, we expand upon the definition of $z$-ideals in rings, as introduced in \cite{Mas73}, and extend it to semirings. We study various properties of these ideals. Our approach involves the utilization of a closure operator, referred to as `$z$-closure,' defined on all ideals within a semiring (see Definition \ref{defcl}). This operator provides an alternative definition for a $z$-ideal. Additionally, the introduction of these closure operators serves the purpose of  examining several distinguished types of $z$-ideals.

An interesting motivation for this work arises from a result in \cite{JRT22}, where it was shown that:
\begin{equation}\label{kppk}
k\text{-prime ideal}\Leftrightarrow \text{prime ideal} +k\text{-ideal,}
\end{equation}
Here, a $k$-ideal is defined according to \cite{H58}, and a $k$-prime ideal is obtained by restricting the definition of a prime ideal to $k$-ideals.

Motivated by (\ref{kppk}), we introduce various types of $z$-ideals in semirings. Using our $z$-closure operators, we establish the following equivalence formulations:
\begin{equation}\label{kxkx}
z\text{-($\star$) ideal}\Leftrightarrow z\text{-ideal} + (\star)\; \text{ideal},
\end{equation}
where `$(\star)$' stands for: maximal, prime, semiprime, irreducible, and strongly irreducible.
 
For instance, to prove properties related to prime $z$-ideals, it suffices to prove them for $z$-prime ideals (see Definition \ref{kpri}). Remarkably, we find that all the distinguished types of $z$-ideals considered in this paper adhere to the equivalent formulation (\ref{kxkx}).

We now briefly describe the content of the paper. In \textsection \ref{prlm}, we gather the required background on semirings and their ideals. In \textsection \ref{zid}, we introduce the notion of a $z$-closure operation and discuss various properties of these operators. We study some properties of $z$-ideals of semirings. We provide some examples of $z$-ideals of semirings. In \textsection \ref{sczi}, we introduce a series of new types of $z$-ideals, namely: $z$-maximal, $z$-prime, $z$-semiprime,  $z$-irreducible, and $z$-strongly irreducible ideals. Using the $z$-closure introduced in \textsection \ref{zid}, we prove the equivalence formulation (\ref{kxkx}) for all these types of $z$-ideals. 

Here are the main results that we prove. We establish the equivalence between $z$-semiprime and $z$-radical ideals (Theorem \ref{spkr}). We provide a representation of any $z$-ideal in terms of $z$-irreducible ideals (Theorem \ref{rpir}). We demonstrate a relationship between $z$-prime, $z$-semiprime, and $z$-strongly irreducible ideals (Theorem \ref{vpss}). We provide a (partial) characterization of arithmetic semirings (Theorem \ref{pcas}). In a $z$-Noetherian semiring (see Definition  \ref{znsr}), we present a representation of $z$-ideals in terms of $z$-strongly irreducible ideals (Proposition \ref{nssi}).

\section{Preliminaries} \label{prlm}

In this section, we provide a brief overview of some key definitions and results concerning semirings and their ideals. For detailed proofs and further exploration of these topics, we refer to \cite{G99}. A \emph{commutative semiring} is a system $( S,+,0,\cdot, 1)$ such that
(i) $( S,+,0)$ is a commutative monoid,
(ii)  $(S, \cdot,1)$ is a commutative monoid,
(iii)  $0s=0,$ for all $s\in  S,$ and
(iv)  $a\cdot(b+c)=a\cdot b +a\cdot c,$ for all $a$, $b$, and $c\in \s$.
We  write $ab$ for $a\cdot b$, and assume that all our semirings are  commutative. 
An \emph{ideal} $\ma$ of a semiring $\s$ is a nonempty  subset of $ S$ satisfying the conditions:
(i) $a+b\in \ma$ and (ii) $sa\in \ma$,
for all $a,$ $b\in \ma$, and for all $s\in  S.$ By $\id$, we denote the set of all ideals of $S$. We denote the largest ideal of $S$ as $S$ itself and the zero ideal as $0$. An ideal $\ma$ is called \emph{proper} if $\ma\neq S$. We denote the set of proper ideals of a semiring $S$ by $\idp$. If $\{\ma_{i}\}_{i\in I}\in \id$, then $\bigcap_{i\in I} \ma_{i}\in \id$. The \emph{sum} of a family $\{\ma_{i}\}_{i\in I}$ of ideals of  $S$ is defined by
\[\label{sum}
\sum_{i\in I} \ma_{i}=\left\{\sum_{j=1}^n x_{i_j} \mid x_{i_j} \in \ma_{i_j}, n\in \mathds{N}\right\},
\] 
which is also an ideal of $ S$. If $\ma$ and $\mb$ are two ideals of $ S$, then their \emph{product} $\ma\mb$ is the ideal generated by the set $\{a b\mid a\in \ma, b\in \mb\}.$ The following self-evident fact about ideals in semirings will prove to be valuable.

\begin{lemma}
\label{psi}
If $\ma$ and $\mb$ are two ideals of a semiring $S$, then $\ma\mb\subseteq \ma \cap \mb$.
\end{lemma}

A semiring $S$ is called \emph{arithmetic} if $\id$ forms a distributive lattice. If $\ma$ and $\mb$ are ideals of $S$,  then the \emph{ideal quotient} or \emph{colon ideal} of $\ma$ over $\mb$ is defined by \[(\ma : \mb)=\{s\in S \mid s\mb\subseteq \ma\}.\]
The \emph{radical} $\sqrt{\ma}$ of an ideal $\ma$ of $S$ is defined by 
\[\label{rada}
\sqrt{\ma}=\left\{ s\in S\mid s^n\in \ma,\;\text{for some}\; n\in \mathds{N}_{>0}\right\}.
\]
It is easy to verify that $\ma\subseteq \sqrt{\ma}$ and $\sqrt{\ma}$ is also an ideal of $ S$. A proper ideal $\ma$ is called a \emph{radical ideal} if $\ma=\sqrt{\ma}.$  A proper ideal $\p$ of  $S$ is called \emph{prime} if $ab\in \p$ implies that $a\in \p$ or $b\in \p$, for all $a,$ $b \in S.$ By $\spec$, we denote the set of all prime ideals of a semiring $S$. It is well-known (see \cite{N18, I59}) that the  radical $\sqrt{\ma}$ of an ideal $\ma$ of a semiring $S$ has the following representation:
\begin{equation}\label{rep}
\sqrt{\ma}=\bigcap_{\p\in \spec} \left\{ \p \mid \ma\subseteq \p \right\}.
\end{equation} 
A proper ideal $\q$ of a semiring $S$ is called \emph{semiprime} if $\ma^2\subseteq \q$ implies $\ma\subseteq \q,$ for all $\ma\in \id$. 
A proper ideal $\m$ of  $S$ is said to be \emph{maximal} if $\m$ is not properly contained in any other proper ideals of $S$. We denote the set of all maximal ideals of $\s$ by $\mx$. It is well-known (see \cite[Corollary 7.13]{G99}) that every maximal ideal in a semiring is prime. The \emph{Jacobson radical} of a semiring $S$ is defined by \[\jr=\bigcap \left\{\m\mid \m\in \mx\right\}.\]  A semiring $S$ is called \emph{semisimple} if $\jr=0.$ A \emph{multiplicatively closed} subset of a semiring $S$ is a subset $X$ of  $S$ such that $1\in X$ and $X$ is closed under multiplication. 

\section{$z$-ideals}
\label{zid}

Since by \cite[Theorem 2]{SZ55}, the set $\mx$ is nonempty, we can define the notion of a $z$-ideal of a semiring.

\begin{definition}
An ideal $\mz$ of a semiring $S$ is called a \emph{z-ideal} if whenever $x\in \mz$, then \[\bigcap\left\{  \m\in \mx\mid x\in \m\right\}\subseteq \mz.\] 	
\end{definition}

By $\z$, we denote the set of all $z$-ideals of a semiring $S$.
Define \[\mathcal{M}_a:=\left\{ \m\in \mx \mid a\in \m\right\}.\] Then the following proposition gives an alternative definition of a $z$-ideal of a semiring.

\begin{proposition}\label{ald}
An ideal $\mz$ of a semiring $S$ is a $z$-ideal if and only if $\mathcal{M}_a=\mathcal{M}_b$ and $b\in \mz $ implies that $a\in \mz$, for all $a$, $b\in S$.
\end{proposition}

\begin{proof}
Suppose that $\mz$ is a $z$-ideal and $\mathcal{M}_a=\mathcal{M}_b$ with $b\in \mz, $ for some $a$, $b\in S$. Since $\mathcal{M}_a=\mathcal{M}_b$, we have $a\in \m$, for all $\m\in \mathcal{M}_b$.  Since $b\in \mz$, this implies that $\bigcap \mathcal{M}_b\subseteq \mz$, hence $a\in \mz$. For the converse, let us assume that for every $a \in S,$ $b \in \z$ satisfying
$\mathcal{M}_a = \mathcal{M}_b$, then $a\in \z$. Now take elements $a\in \z$ and $\bigcap\mathcal{M}_a$.
This means that $\mathcal{M}_a\subseteq \mathcal{M}_x$. From this fact, we can infer
that $\mathcal{M}_x = \mathcal{M}_x \cup \mathcal{M}_a = \mathcal{M}_{ax}$. Since $\mz$ is an ideal and $a\in \mz$,
it follows that $ax \in \mz$. Now the equality $\mathcal{M}_x = \mathcal{M}_{ax}$ and the
assumption imply $x\in \mz$, proving that $\mz$ is a $z$-ideal.
\end{proof}

\begin{remark}
It is easy to see that the above equivalent condition can further be relaxed, namely: $\mz$ is a $z$-ideal if and only if $\mathcal{M}_a\supseteq \mathcal{M}_b$ and $b\in \mz $ implies that $a\in \mz$.
\end{remark}

In the next proposition, we gather some elementary properties of $z$-ideals of a semiring. 

\begin{proposition}\label{epzi}
Let $\s$ be a semiring.
\begin{enumerate}
		
\item\label{ainz} If $\{\mz_{i}\}_{i\in I}$ is a family of $z$-ideals of $S$, so is their intersection.
		
\item The Jacobson radical $\mathcal{J}(S)$ of $S$ is a $z$-ideal.
		
\item If $S$ has a unique maximal ideal $\m$ and if $\ma\in \id$ such that $\ma\subsetneq \m$, then $\ma$ is not a $z$-ideal.
\end{enumerate}	
\end{proposition}

\begin{proof}
(1) Let $x\in \bigcap_{i\in I}\ma_i$. Then $x\in \ma_i$ for every $i\in I$. Since $\ma_i\in \z$, we have $\bigcap\mathcal{M}_x\subseteq \ma_i$ for every $i\in I$, and thus $\bigcap \mathcal{M}_x\subseteq \bigcap_{i\in I}\ma_i$.
	
(2) Since every maximal ideal is a $z$-ideal, the claim now follows from (\ref{ainz}).
	
(3) Since $\ma\subsetneq \m$, there exists $x\in \m$, but $x\notin \ma$. Suppose that $y\in \ma$. Since $\m$ is the only maximal ideal of $S$, we must have $\mathcal{M}_x=\mathcal{M}_y$ and $y\in \ma$, whereas $x\notin \ma$. This implies that $\ma$ is not a $z$-ideal of $S$.
\end{proof}

The next technical proposition will play a significant role later on. It generalizes \cite[Theorem 1.1]{Mas73}. Before we state it, we need a lemma. The proof of this lemma relies on Zorn's lemma and is identical to that of rings, and hence, we omit it.

\begin{lemma}\label{mds}
Let $S$ be a  semiring, $\ma$ be an ideal of $S$, and $T$ be a multiplicatively closed set of $S$ that is disjoint from $\ma$. Then, there exists an ideal $\mb$ of $S$ that is maximal with respect to containing $\ma$ and being disjoint from $T$. Moreover, every such ideal $\mb$ is prime.
\end{lemma}

\begin{proposition}
\label{mpz} If $\p$ is a prime ideal of a semiring $S$ which is minimal with
respect to the property to contain a given  $z$-ideal $\ma$ of $S$, then $\p\in \z$.
\end{proposition}

\begin{proof}
Applying Lemma \ref{mds}, the proof is the same as given in \cite[Theorem 1.1]{Mas73} for rings.
\end{proof}

Our next goal is to determine when $z$-ideals of semirings are closed under finite products. We say an ideal of the form $\mathsf{m}_x:=\bigcap \mathcal{M}_x$ ($x\in S$) of a semiring $S$ a \emph{basic $z$-ideal}; $\mathsf{m}_x$ indeed a $z$-ideal follows from Proposition \ref{epzi}(\ref{ainz}) and the fact that every maximal ideal ideal is a $z$-ideal.

\begin{theorem}\label{prdz}
The product of two $z$-ideals of a semiring $S$ is a $z$-ideal if and only if every basic $z$-ideal of $S$ is idempotent.
\end{theorem}

\begin{proof}
Let us suppose that the product of two $z$-ideals of $S$ is a $z$-ideal. Let $x\in S$. We need to show that $\mathsf{m}_x^2=\mathsf{m}_x$. It is clear that $\mathsf{m}_{x^2}=\mathsf{m}_x$. Since $x\in \mathsf{m}_x$, we have $x^2\in \mathsf{m}_x^2$. Since $\mathsf{m}_x$ is a $z$-ideal, by assumption, $\mathsf{m}_x^2$ is also a $z$-ideal. Moreover, $\mathsf{m}_{x^2}\subseteq \mathsf{m}_x^2.$ From these, we obtain
\[ \mathsf{m}_x=\mathsf{m}_{x^2}\subseteq \mathsf{m}_x^2\subseteq \mathsf{m}_x,\]
implying that $\mathsf{m}_x^2=\mathsf{m}_x.$ For the converse, suppose that every basic $z$-ideal of $S$ is idempotent. Let us consider two $z$-ideals $\ma$ and $\mb$ of $S$. Let $y\in \ma\mb$ for some $y\in S$. Then $\mathsf{m}_y\subseteq \ma$ and $\mathsf{m}_y\subseteq \mb$. Applying the hypothesis, 
\[\mathsf{m}_y=\mathsf{m}_y\mathsf{m}_y\subseteq \ma\mb,\]
proving that $\ma\mb$ is a $z$-ideal of $S$.
\end{proof}

\begin{definition}
If every basic $z$-ideal of a semiring $S$ is idempotent, then we say $S$ is a \emph{bzi-semiring}.
\end{definition}

As an example, notice that every idempotent semiring is a \emph{bzi}-semiring. It is easy to see that not every semiring is a  \emph{bzi}-semiring. For example, consider the ring $\mathds{Z}$ of integers (and hence a semiring). The $z$-ideal $\mathsf{m}_3$ is nothing but $3\mathds{Z}$, however, $3\mathds{Z}\cdot 3\mathds{Z}=9\mathds{Z}\neq 3\mathds{Z}$.

\begin{example}
\label{exm}~
Now, it may be appropriate to consider some examples of 
$z$-ideals in semirings.
	
(1)\label{exm1} Every maximal ideal of a semiring $S$ is a $z$-ideal. Indeed: for $\m\in \mx$, $\mathcal{M}_a=\mathcal{M}_b$, and $a\in \m$ implies that $\m\in \mathcal{M}_a=\mathcal{M}_b$. Therefore, $b\in M$, and hence $\m$ is a $z$-ideal.
	
(2) Consider the semiring $\mathcal{C}\left(X, \mathds{I}\right)$ of all continuous functions defined on a topological space $X$ and taking the values in a topological semiring $\mathds{I}=\left([0,1], \cdot, \vee\right )$, with pointwise operations: addition $\vee$ and multiplication $\cdot$
of functions. Then for every prime ideal $\p$ of $\mathcal{C}\left(X, \mathds{I}\right)$, \[ \mathcal{O}_{\p}:=\left\{
f\in \mathcal{C}(X, \mathds{I})\mid fg=0,\;\text{for some}\; g\in \mathcal{C}(X, \mathds{I})\setminus \p \right\}\]
is a $z$-ideal of $\mathcal{C} \left(X, \mathds{I}\right)$ (see \cite[\textsection 4]{VMS19}). Moreover, every minimal prime ideal of $\mathcal{C}\left(X, \mathds{I}\right)$ is a $z$-ideal (see \cite[Proposition 4.7]{VMS19}).
	
(3) If $S$ is a semisimple semiring, then $0$ is a $z$-ideal of $S$ (see Lemma \ref{lclk}(\ref{ssz})).
	
(4) The minimal prime ideals of a semiring are $z$-ideals (see Proposition \ref{mpz}).
\end{example}

By employing ideal quotients, we derive additional examples of $z$-ideals of semirings. This result is presented in the following proposition, which serves as a generalization of \cite[Proposition 1.3]{Mas73}, along with its corollary.

\begin{proposition}\label{icj}
If $\ma$ is a $z$-ideal and $\mb$ is an ideal of a semiring $S $, then $(\ma:\mb)$ is a $z$-ideal of $S$.
\end{proposition}

\begin{proof}
Suppose that $\mathcal{M}_x\supseteq \mathcal{M}_y$ and $y\in (\ma : \mb).$ This implies $\mathcal{M}_{xs}\supseteq \mathcal{M}_{ys}$ and $ys\in \ma$ for all $s\in \mb$. Since $\ma$ is a $z$-ideal, we have $xs\in \ma$ for all $s\in \mb$. Hence $x\in (\ma : \mb).$
\end{proof}

\begin{corollary}
Suppose that  $\mb$, $\{\mb_{i}\}_{i \in I}$, $\mc$ are ideals and  $\ma$, $\{\ma_{j}\}_{j \in \mb}$ are $z$-ideals of a semiring $S$. Then  $ (\ma : \mb),$  $((\ma : \mb) : \mc),$ $(\ma : \mb \mc),$ $((\ma : \mc) : \mb),$ $ (\bigcap_{j} \ma_{j} : \mb),$ $ \bigcap_{j}(\ma_{j}: \mb),$ $ (\ma : \sum_{i}\mb_{i}),$ and $ \bigcap_{i}(\ma : \mb_{i})$ are all $z$-ideals of $S$.
\end{corollary}

It is worth noting, as established in \cite{SA92}, that the concept of $k$-ideals, also known as subtractive ideals, in semirings can be defined and studied using closure operations. We adopt a similar approach when dealing with $z$-ideals.

\begin{definition}\label{defcl}
Let $S$ be a semiring.  The \emph{$z$-closure} operation on $\id $  is defined by
\begin{equation}
\label{clkdef}
\cz(\ma):=\bigcap_{\mz \in \z}\left\{\mz \mid \mz\supseteq \ma\right\}.
\end{equation} 	
\end{definition}

\begin{remark}
What we refer to as the $z$-closure operator $\cz$(-) was originally introduced as $I_z$ in \cite[p. 281]{Mas73} to describe the smallest $z$-ideal containing the ideal $I$. Our primary objective is to utilize this closure operator $\cz$(-) to establish properties of various distinguished types of $z$-ideals (see \textsection \ref{sczi}). It is worth noting that the $z$-closure can also be expressed as $\cz(\ma)=\mathcal{I}\mathcal{V}(\ma),$ where $\mathcal{V}(\ma):=\left\{ \mz \in \z \mid \mz \supseteq \ma\right\}$ and $\mathcal{I}(X):=\bigcap X$ $(X\subseteq S)$.  This formulation, denoted as $\mathcal{I}\mathcal{V}$(-), stands in contrast to $\mathcal{V}\mathcal{I}$(-), which represents the closed sets of a Zariski topology defined on the prime spectrum of a (commutative) ring.
\end{remark}

The next lemma gathers some of the elementary properties of $z$-closure operations. Among them, (\ref{iclk}), (\ref{ckr}), (\ref{rzi}), and (\ref{czan}) generalize results in \cite[p.\,281]{Mas73}; (\ref{cabcc}) generalizes \cite[Proposition 3.1(a)]{Mas80} and \cite[Lemma 2.2(1)]{Ben21}; (\ref{arbin}) generalizes \cite[Lemma 1.3(c)]{AAT13} and \cite[Lemma 2.2(2)]{Ben21}; (\ref{rzi})--(\ref{sca})  generalize commutative versions of  \cite[Proposition 3.7, Proposition 3.9]{MJ20}.

\begin{lemma}\label{lclk}
Let $S$ be a semiring and let  $\ma$, $\{\ma_{i}\}_{i \in I}$, and $\mb$ be ideals of $S $. Then the following hold.
\begin{enumerate}
		
\item\label{iclk} $\cz(\ma)$ is the smallest $z$-ideal containing $\ma$.

\item \label{ckr}
$ \cz(S) =S .$
		
\item \label{ssz}
If $S$ is semisimple, then $\cz(0)=0.$
		
\item\label{clcl} $\cz(\cz(\ma))=\cz(\ma).$
		
\item\label{ijcl} If $\ma\subseteq \mb$, then $\cz(\ma)\subseteq \cz(\mb).$
		
\item \label{clu}
$\cz(\langle \ma\cup \mb\rangle )\supseteq \cz(\ma) \cup \cz(\mb).$

\item\label{cabcc} $\cz(\ma+\mb)=\cz(\cz(\ma)+\cz(\mb)).$
		
\item\label{altd} $\ma$  is a $z$-ideal if and only if $\ma=\cz(\ma).$
		
\item\label{rzi} $\sqrt{\ma}\subseteq \cz(\ma).$

\item\label{czsq} $\cz(\sqrt{\ma})=\cz(\ma)$.

\item If $\ma$ is a $z$-ideal, then $\cz(\sqrt{\ma})=\ma$.

\item\label{sca} $\sqrt{\cz(\ma)}=\cz(\sqrt{\ma})$.

\item\label{arbin} $\cz(\ma \mb)=\cz\left(\ma\cap \mb\right)=  \cz (\ma)\cap \cz(\mb).$

\item $\cz(\ma\mb)=\cz(\ma\cz(\mb))=\cz(\cz(\ma)\cz(\mb)).$

\item\label{czan} $\cz(\ma^n)=\cz(\ma)$, for any $n\in \mathds{N}_{>0}$.

\item\label{clijk} If $S$ is a bzi-semiring, then $\cz(\ma\mb)= \cz(\ma)\cz(\mb).$
\end{enumerate}
\end{lemma}

\begin{proof}
(1) From the Definition \ref{clkdef}, it is clear that $\ma \subseteq \cz(\ma).$ Suppose that $\{\mz_{i}\}_{i\in I}$ is a family of $z$-ideals of $S$. Let $x\in \bigcap_{i\in I}\mz_{i}$. Then $x\in \mz_{i}$ for each $i \in I$. This implies that $\bigcap \mathcal{M}_{x,i}\subseteq \mz_{i}$, for each $i \in I$, and hence $\bigcap_{i\in I} \mathcal{M}_{x,i}\subseteq \bigcap_{i\in I} \mz_{i}$, proving that $\bigcap_{i\in I}\mz_{i}$ is also a $z$-ideal. If $\mz'$ is a $z$-ideal of $S$ such that $\ma\subseteq \mz'$, then obviously \[\mz'\supseteq \bigcap_{\mz \in \z} \left\{ \mz\mid \mz\supseteq \ma\right\}.\] This proves that $\cz(\ma)$ is the smallest $z$-ideal containing $\ma$. 

(2) Straightforward.

(3) It is sufficient to show that if $S$ is semisimple, then $0$ is a $z$-ideal. Suppose that $\mathcal{M}_a=\mathcal{M}_b$ and $a\in 0$. Then $a=0$ and therefore $\mathcal{M}_a=\mathcal{M}_0=\mx$, implying $\mathcal{M}_b=\mx$. Hence, $b\in \bigcap \mx=0$, that is, $b=0$. By Proposition \ref{ald}, this implies that $0$ is a $z$-ideal.

(4) By (\ref{iclk}), it follows that $\cz(\ma)\subseteq \cz(\cz(\ma)).$ Conversely, let $x\in \cz(\cz(\ma)).$ Then \[x\in \bigcap_{\mz\in \z} \left\{ \mz\mid \cz(\ma)\subseteq \mz\right\},\] and this implies that $x\in \mz$ for all $\mz$ with the property: $\cz(\ma)\subseteq \mz$. From this, we obtain \[x\in \bigcap_{\mz\in \z} \left\{ \mz\mid \ma\subseteq \mz\right\}=\cz(\ma).\]

(5) Straightforward.

(6) Follows from (\ref{ijcl}).

(7) Observe that $\ma+\mb\subseteq \cz(\ma+\mb)$ and  $\ma+\mb\subseteq \cz(\cz(\ma)+\cz(\mb)).$ Since by (\ref{iclk}), $\cz(\ma+\mb)$ is the smallest $z$-ideal containing $\ma+\mb$, we must have \[\cz(\ma+\mb)\subseteq \cz(\cz(\ma)+\cz(\mb)).\] Conversely, $\cz(\ma+\mb)\supseteq \ma,$ $\mb;$ and hence, by (\ref{clcl}) we have $\cz(\ma+\mb)\supseteq \cz(\ma)+ \cz(\mb).$ Applying (\ref{clcl}) and (\ref{ijcl}) in the last containment, we obtain the desired claim.

(8) If $\ma=\cz(\ma)$, then by (\ref{iclk}), $\mz$ is a $z$-ideal. Conversely, if $\ma$ is a $z$-ideal, then by Definition \ref{clkdef}, we have the equality.

(9) Suppose that $x\in \sqrt{\ma}$ and $\ma\subseteq \mz$ for some $\mz\in \z$. This implies $x^n\in I\subseteq \mz$, for some $n\in \mathds{N}_{>0}.$ Since $\mz$ is a $z$-ideal, this means that \[x\in \mathcal{M}_x=\mathcal{M}_{x^n}\subseteq \mz.\] Hence, $x\in \cz(\ma).$

(10) Since $\ma\subseteq \sqrt{\ma},$ by (\ref{ijcl}) we have $\cz(\sqrt{\ma})\supseteq \cz(\ma).$ The other half follows from (\ref{rzi}) and (\ref{ijcl}).

(11) Follows from (\ref{czsq}) and (\ref{altd}).

(12) Since $\cz(\ma)$ is a $z$-ideal, by applying (\ref{rzi}), we obtain \[ \cz(\ma)\subseteq \sqrt{\cz(\ma)}\subseteq \cz(\cz(\ma))=\cz(\ma).\] Hence $\sqrt{\cz(\ma)}=\cz(\ma).$ By (\ref{czsq}), we then have the desired equality.

(13) We show $\cz(\ma\mb)=\cz(\ma)\cap \cz(\mb)$, and from that we have the other equality for free. Since  $\ma\mb\subseteq \ma$, $\ma\mb\subseteq  \mb$, we have \[\cz(\ma\mb)\subseteq  \cz(\ma\cap \mb)\subseteq  \cz(\ma)\cap \cz(\mb)\in \z.\]
Therefore, to obtain the desired identities, it is sufficient to show that $\cz(\ma)\cap \cz(\mb)$ is the smallest $z$-ideal containing $\ma\mb$. Let us suppose that $\ma\mb \subseteq \mz$ for some $\mz \in \z$. Let $\mathrm{Min}_{\mz}(S)$ be the set of all minimal prime ideals of $S$ containing $\mz$. If $\p \in \mathrm{Min}_{\mz}(S)$, then $\p \in \z$ by Proposition \ref{mpz}. Since
\[ \sqrt{\mz}=\sqrt{\cz(\mz)}=\cz(\sqrt{\mz})=\sqrt{\cz(\mz)}\subseteq \cz(\cz(\mz))=\cz(\mz)=\mz,\]
by Proposition \ref{epzi}(\ref{ainz}) we have 
\[ \mz = \bigcap_{\p \in \mathrm{Min}_{\mz}(S)} \p.\]
Here we have used the property that the radical of an ideal in a semiring can be represented as the intersection of minimal prime ideals containing it. The proof of this fact is identical to that of rings.
Since for each $\p \in \mathrm{Min}_{\mz}(S)$, we have $\ma \mb \subseteq \p$; from this, we infer that $\ma \subseteq \p$ or $\mb \subseteq \p$. It follows that $\cz(\ma)\cap \cz(\mb)\subseteq \p.$ Hence $\cz(\ma)\cap \cz(\mb)\subseteq \mz,$ that is, $\cz(\ma)\cap \cz(\mb)$ is the smallest $z$-ideal containing $\ma\mb$.

(14) From the following string of relations:
\begin{align*}
\cz(\ma\mb)&\subseteq \cz(\ma\cz(\mb))\\&\subseteq \cz(\cz(\ma)\cz(\mb))\\&=\cz(\cz(\ma)\cap \cz(\cz(\mb)))\\&=\cz(\ma)\cap \cz(\mb)\\&=\cz(\ma\mb),
\end{align*}
we have the two equalities.

(15) Since $\ma^n\subseteq \ma$, for any $n\in \mathds{N}$, by (4), we have $\cz(\ma^n)\subseteq \cz(\ma).$ For the converse, suppose that $x\in \ma$. Then $x^n\in \ma^n$, and $\ma^n\subseteq \cz(\ma^n)$ by (1). Since  $\cz(\ma^n)$ is a $z$-ideal, this implies that $\bigcap \mathcal{M}_{x^n}\subseteq \cz(\ma^n).$ Thus
\[ x\in \bigcap \mathcal{M}_{x}=\bigcap \mathcal{M}_{x^n}\subseteq \cz(\ma^n).\]
Hence $\ma\subseteq \cz(\ma^n),$ and applying (\ref{clcl}) gives the desired containment. 

(16)  By (\ref{iclk}), we have $\ma\subseteq \cz(\ma)$ and $\mb\subseteq \cz(\mb)$, and hence $\ma\mb\subseteq \cz(\ma)\cz(\mb)$. Since $S$ is a \emph{bzi}-semiring, by Theorem \ref{prdz}, we have $\cz(\ma)\cz(\mb)\in \z$. Therefore, it is now sufficient to show that $\cz(\ma)\cz(\mb)$ is the smallest $z$-ideal that contains $\ma\mb$. The rest of the argument in the proof is similar to (\ref{arbin}).
\end{proof}

So far, we have discussed the closedness of $z$-ideals in semirings under intersections and products. The following proposition provides equivalent criteria for the closedness of $z$-ideals under sums. This result extends \cite[Proposition 3.1(b)]{Mas80}, and its proof is identical to the one given there.

\begin{proposition}\label{cujo}
In a semiring $S$, the following are equivalent.
\begin{enumerate}
\item\label{cuj} If $\ma$, $\mb\in \z$, then $\ma+ \mb \in \z$.
		
\item If $\ma$, $\mb\in \id$, then $\cz(\ma+\mb)=\cz(\ma)+ \cz(\mb).$
		
\item If $\{\ma_i\}_{i\in \Lambda}\subseteq \z$, then $\sum_{i\in \Lambda} \ma_i\in \z.$
		
\item If $\{\ma_i\}_{i\in \Lambda}\subseteq \id$, then $\cz\left( \sum_{i\in \Lambda} \ma_i\right)=\sum_{i\in \Lambda} \cz(\ma_i)$.
\end{enumerate}
\end{proposition}

\section{Some distinguished types of $z$-ideals}
\label{sczi}

This section is dedicated to the examination of specific distinguished types of $z$-ideals of semirings. These types of $z$-ideals are derived by `restricting' the conventional definitions of corresponding types of ideals in semirings to $z$-ideals.

\begin{definition}
A proper $z$-ideal of a semiring  $S$ is called \emph{$z$-maximal} if it is not properly contained in another proper $z$-ideal. 	
\end{definition}

The following result guarantees that the set of $z$-maximal ideals in a semiring is nonempty. This result extends \cite[Corollary 2.2]{SA93}

\begin{lemma}\label{exma} 
Every proper $z$-ideal of a semiring  $S$ is contained in a $z$-maximal ideal of $S$.
\end{lemma}

\begin{proof}
We first prove the claim when $S$ is finitely generated. Suppose that $S=\langle s_1, \ldots, s_n\rangle,$ where $s_i\in S$, for all $1\leqslant i \leqslant n$. Let $\ma$ be a $z$-ideal of $S$. Consider the poset
\[\mathcal{E}:=\left\{ \mz\in \z \mid \ma\subseteq \mz \subsetneq S\right\}.\]
Since $\ma \in \mathcal{E}$, the set $\mathcal{E}$ is nonempty. If $\mathcal{K}:=\{\mz_i\}_{i\in I}\subseteq \mathcal{E}$ is a chain, then $\mz':=\bigcup_{i\in I} \mz_i$ is a $z$-ideal. Moreover, $\mz'\neq S$. Hence $\mz'\in \mathcal{E}$. So by Zorn's lemma, $\mathcal{E}$ has a maximal element, which is the desired $z$-maximal ideal. Now the claim follows by taking $S=\langle 1\rangle$.
\end{proof}

As pointed out in the introduction, the equivalent formulation (\ref{kxkx}) holds for all the distinguished types of $z$-ideals in the semirings we are examining. The following proposition serves as the initial example of this equivalence.

\begin{proposition}
\label{eqsm}
An ideal $\m$ of a semiring $S$ is $z$-maximal if and only if $\m$ is a maximal $z$-ideal of $S$.
\end{proposition}

\begin{proof}
It is clear that every $z$-ideal of $S$ which is also a maximal ideal is a $z$-maximal ideal. For the converse, suppose that  $\m$ is a $z$-maximal ideal of $S$ and $\m\subsetneq \ma\subsetneq S,$ for some $\ma\in \id\setminus \z .$ We need to show that $\m$ is a maximal ideal of $S$. From the assumption, we have
\[\m=\cz(\m)\subsetneq \ma \subsetneq \cz(\ma)\subsetneq \cz(S) =S,
	\]
where the first and last equalities follow respectively from Lemma \ref{lclk}(\ref{altd}) and Lemma \ref{lclk}(\ref{ckr}), and the first and the last strict  inclusions follow from Lemma \ref{lclk}(\ref{ijcl}). The middle strict inclusion follows from  Lemma \ref{lclk}(\ref{altd}) and the assumption that $\ma\in \id\setminus \z$.
Since $\m$ is $z$-maximal and since $\cz(\ma)$ is a $z$-ideal, $\m \subsetneq \cz(\ma)\subsetneq S$ leads to a contradiction. Therefore $\m$ is a maximal ideal of $S$.
\end{proof}

Our next objective is to introduce the concepts of $z$-prime and $z$-semiprime ideals of semirings. Additionally, we will introduce the notion of $z$-radicals of $z$-ideals.

\begin{definition}\label{kpri}
A proper $z$-ideal $\p$ of a semiring $S$ is called \emph{$z$-prime}  if  $\ma\mb\subseteq \p$ implies  $\ma\subseteq \p$ or $\mb\subseteq \p$, for all $\ma,$ $\mb\in \z.$ We denote by $\mathrm{Spec}_z(S)$ the set of all $z$-prime ideals of $S$ 
\end{definition}

We now show the equivalence formulation (\ref{kxkx}) for the $z$-prime ideals of semirings.
This result partially extends \cite[Proposition 3.5]{JRT22} 

\begin{proposition}
\label{eqp} 
An ideal $\p$ of a bzi-semiring $S$ is $z$-prime if and only if $\p$ is a prime $z$-ideal of $S$.
\end{proposition}

\begin{proof}
It is evident that if $\p$ is a prime $z$-ideal, then obviously $\p$ is $z$-prime ideal.
For the converse, let $\p$ be a $z$-prime ideal of $S$ and $\ma \mb\subseteq \p$, for some $\ma$, $\mb\in \id$.  This  implies that
\[\cz(\ma) \cz(\mb) = \cz(\ma\mb) \subseteq \cz(\p))=\p,\]
where the first equality follows from Lemma \ref{lclk}(\ref{clijk}).
Since $\p$ is a $z$-prime ideal, we must have  $\ma\subseteq  \cz(\ma)\subseteq \p$ or $\mb\subseteq\cz(\mb)\subseteq \p.$ This proves that $\p$ is a prime ideal of $S$.
\end{proof}

The existence of minimal prime ideals in semirings is a well-established fact (see \cite[Proposition 7.14]{G99}). In the next proposition, we extend this result to $\mathrm{Spec}_z(S)$.

\begin{proposition}
\label{kmkp}
Every nonzero bzi-semiring contains a minimal $z$-prime ideal.
\end{proposition}

\begin{proof}
Suppose that  $S$ is a nonzero semiring. It follows from Lemma \ref{exma} that $S$ has a $z$-maximal ideal, say $\m$, which by Proposition \ref{eqsm} is also a  maximal $z$-ideal, and hence $\m$ is a prime $z$-ideal by \cite[Corollary 7.13]{G99}. Finally, by  Proposition \ref{eqp}, we conclude that $\m$ is a $z$-prime ideal. Hence, the set $\mathrm{Spec}_z(S)$ is nonempty. The claim now follows from a routine application of Zorn's lemma.	
\end{proof}
 
The following proposition extends (\cite[Lemma 3.19]{AK13}). The proof is identical to rings.
 
\begin{proposition}[Prime Avoidance Lemma]
Let $\ma$ be a subset of a semiring $S$ that is stable under addition and multiplication, and $\p_1, \ldots, \p_n$ be $z$-ideals such that $\p_3, \ldots, \p_n$
are $z$-prime ideals. If $\ma\nsubseteq \p_j$, for all $1\leqslant j\leqslant n,$ then there is an $x \in \ma$ such that $x \notin
\p_j$, for all $1\leqslant j\leqslant n$.
\end{proposition}

We will now introduce the concept of $z$-semiprime ideals. In Theorem \ref{spkr}, we shall demonstrate the equivalence between $z$-semiprime ideals and $z$-radical ideals (see Definition \ref{drad}).

\begin{definition}\label{kspr}
A proper $z$-ideal $\p$ of a semiring $S$ is called \emph{$z$-semiprime} if $\ma^2\subseteq \p$ implies $\ma\subseteq \p,$ for all $\ma\in \z$. 
\end{definition}

Much like with $z$-prime ideals, we also have an equivalent formulation for $z$-semiprime ideals.

\begin{proposition}\label{eqpp}
An ideal $\q$ of a bzi-semiring $S$ is $z$-semiprime if and only if $\q$ is a $z$-ideal and a semiprime ideal of $S$.
\end{proposition}

\begin{proof}
If $\q$ is a $z$-ideal and also a semiprime ideal of a semiring $S$, then by Definition \ref{kspr}, $\q$ is $z$-semiprime. For the converse, suppose that $\q$ is a $z$-semiprime ideal and $\ma^2\subseteq \q$ for some $\ma\in\id.$ Then
\[(\cz(\ma))(\cz(\ma))\subseteq \cz(\ma^2)\subseteq \cz(\q)=\q,\]
where the first and second inclusions respectively follow from Lemma \ref{lclk}(\ref{clijk}) and Lemma \ref{lclk}(\ref{ijcl}), whereas the last equality follows from the fact that $\q$ is a $z$-ideal and Lemma \ref{lclk}(\ref{ijcl}). Since $\q$ is $z$-semiprime, $\cz(\ma)\subseteq \q$, and by  Lemma \ref{lclk}(\ref{iclk}), we have $\ma\subseteq \q.$ 
\end{proof}

In the context of rings, the radical of an ideal $\ma$ can be defined as the intersection of all prime ideals that contain $\ma$. If we narrow our focus to $z$-ideals and consider prime ideals as $z$-primes, we arrive at the concept of a $z$-radical.
 
\begin{definition}
\label{drad}
	
The \emph{$z$-radical of a $z$-ideal $\ma$} of a semiring $S$ is defined by
\begin{equation}\label{radki}
\sqrt[z]{\ma}=\bigcap_{\p\in \mathrm{Spec}_z(S)}\left\{\p\mid \ma\subseteq \p \right\}.
\end{equation}
Whenever $\sqrt[z]{\ma}=\ma,$ we say $\ma$ is a \emph{$z$-radical ideal}.	
\end{definition}

In the following lemma, we collect some elementary properties of $z$-radicals of $z$-ideals.

\begin{lemma}\label{radpr}
In the following, $\ma$ and $\mb$ are $z$-ideals of a semiring $S$.
\begin{enumerate}
		
\item\label{irad} $\sqrt[z]{\ma}$ is a $z$-ideal containing $\ma$.
		
\item $\sqrt[z]{\sqrt[z]{\ma}}=\sqrt[z]{\ma}.$
		
\item\label{ijicj} $\sqrt[z]{\ma\mb}\supseteq  \sqrt[z]{\ma\cap \mb}\supseteq  \sqrt[z]{\ma}\cap \sqrt[z]{\mb}.$
\end{enumerate}
\end{lemma}

\begin{proof}
(1) From Definition \ref{drad}, it is clear that $\ma\subseteq \sqrt[z]{\ma}$. To show $\sqrt[z]{\ma}$ is a $z$-ideal, let $x\in \sqrt[z]{\ma}.$ This implies that $x\in \p$, for all $\p\in \mathrm{Spec}_z(S)$ such that $\p\supseteq \ma$. Since each such $\p$ is a $z$-ideal, $\bigcap \mathcal{M}_x\subseteq \p$. This implies that
\[ \bigcap_{\m \in \mx} \{\m \mid x\in \m \}\subseteq \bigcap_{\p \supseteq \ma}\p=\sqrt[z]{\ma}.\]
	
(2) Since by (1), $\sqrt[z]{\ma}$ is a $z$-ideal, by Lemma \ref{lclk}(\ref{iclk}), we have $\sqrt[z]{\ma}\subseteq \sqrt[z]{\sqrt[z]{\ma}}.$ Since by (1), $\ma\subseteq \sqrt[z]{\ma},$ by Definition \ref{drad}, we have $\sqrt[z]{\ma}\supseteq \sqrt[z]{\sqrt[z]{\ma}}.$
	
(3) Follows from Definition \ref{drad} and Lemma \ref{psi}.
\end{proof}

Next, we aim to establish the equivalence between $z$-semiprime ideals and $z$-radical ideals in a \emph{bzi}-semiring. This equivalence is well-known to exist for semiprime and radical ideals in (noncommutative) rings and semirings. In the noncommutative case, the concept of $m$-systems and $n$-systems in rings (see \cite[\textsection 10]{L01}) and in semirings (see \cite[Proposition 7.25]{G99}) are required. However, in the context of $z$-ideals in commutative semirings, multiplicatively closed subsets alone are sufficient. To establish this equivalence, we will proceed through a series of lemmas. 
Note that the first lemma in fact gives a characterization of $z$-prime ideals of semirings, similar to prime ideals of (commutative) rings.

\begin{lemma}\label{rpms}
A $z$-ideal $\p$ of a bzi-semiring $S$ is $z$-prime if and only if $S\setminus \p$ is a multiplicatively closed subset.
\end{lemma}

\begin{proof}
Let  $\p$ be a $z$-prime ideal of $S$. Since  by definition $\p$ is a proper ideal of $S$, $1\notin \p$, and hence $1\in S\setminus \p$. Suppose that $x,$ $y\in S\setminus \p$. Then $xy\notin \p$, by Proposition \ref{eqp}, and hence $xy\in S\setminus \p.$ Conversely, suppose that $\p$ is a proper $z$-ideal of $S$ and let $x,$ $y\notin \p$. This implies that $x,$ $y\in S\setminus \p$. Since $S\setminus \p$ is multiplicatively closed, $xy\in S\setminus \p$, and hence $xy\notin \p$, proving that $\p$ is a prime ideal of $S$. By Proposition \ref{eqp}, this implies that $\p$ is a $z$-prime ideal of $S$. 
\end{proof}

\begin{lemma}\label{mxkp}
Let $S$ be a multiplicatively closed subset of a bzi-semiring $S$. Suppose that  $\p$ is a $z$-ideal which is maximal with respect to the property: $\p\cap S=\emptyset$. Then $\p$ is a $z$-prime ideal.
\end{lemma}

\begin{proof}
Suppose that $x\notin \p$ and $y\notin \p$, but $\langle x\rangle \langle y \rangle \subseteq \p.$ By the assumption on $\p$, there exist $s,$ $s'\in S$ such that $s\in \langle x\rangle+\p$ and $s'\in \langle y\rangle+\p$. This implies that
\[ss'\in (\langle x\rangle+\p)(\langle y\rangle+\p)=\langle x\rangle \langle y \rangle+\p\subseteq \p,\]
a contradiction. Hence $\p$ is prime and by Proposition \ref{eqp}, $\p$ is a $z$-prime ideal.
\end{proof}

\begin{lemma}\label{rkt}
Let $\ma$ be a $z$-ideal of a bzi-semiring $S$. Then   \[\sqrt[z]{\ma}=\mathcal{T}:=\left\{r\in S\mid \mathrm{every\; multiplicatively\; closed\; subset\; containing}\;r\;\mathrm{intersects}\;\ma\right\}.\]
\end{lemma}

\begin{proof}
Suppose that  $r\in \mathcal{T}$ and $\p\in \mathrm{Spec}_z(S)$ such that $\ma\subseteq \p$. Then by Lemma \ref{rpms}, $S\setminus \p$ is a multiplicatively closed subset of $S$ and $r\notin  S\setminus \p.$ Hence $r\in \p$. Conversely, let $r\notin \mathcal{T}$. This implies that there exists a multiplicatively closed subset $X$ of $S$ such that $r\in X$ and $X\cap \ma=\emptyset.$ By Zorn's lemma, there exists a ideal (and hence a $z$-ideal)  $\p$ containing $\ma$ and maximal with respect to the property that $\p\cap X=\emptyset$. By Lemma \ref{mxkp}, $\p$ is a prime ideal and by Proposition \ref{eqp}, $\p$ is a $z$-prime ideal with $r\notin \p$.
\end{proof}

\begin{lemma}\label{mms}
Let $\ma$ be a $z$-semiprime ideal of a semiring $S$ and $x\in S\setminus \ma$. Then there exists a multiplicatively closed subset $X$ of $S$ such that $x\in X\subseteq S\setminus \ma.$
\end{lemma}

\begin{proof}
Define the elements of $X=\{x_1, x_2, \ldots, x_n, \ldots\}$ inductively as follows:
$
x_1:= x;$
$x_2:=x_1x_1;$
$\ldots$;
$x_n:=x_{n-1}x_{n-1};$ $\ldots$. Obviously $x\in S$ and it is also easy to see that $x_i,$ $x_j\in X$ implies that $x_ix_j\in X$.
\end{proof}

\begin{theorem}\label{spkr}
For any $z$-ideal $\ma$ of a bzi-semiring $S$, the following are equivalent.	\begin{enumerate}
		
\item\label{iksp} $\ma$ is $z$-semiprime.
		
\item\label{iinp} $\ma$ is an intersection of $z$-prime ideals of  $S$.
		
\item\label{iradi} $\ma$ is $z$-radical.
\end{enumerate}
\end{theorem}

\begin{proof}
By Definition \ref{drad}, it follows that (3)$\Rightarrow$(2). Since the intersection of $z$-prime ideals is a $z$-prime ideal and every $z$-prime ideal is $z$-semiprime, (2)$\Rightarrow$(1) follows. What remains is to show that (1)$\Rightarrow$(3) and for that it suffices to show   $\sqrt[z]{\ma}\subseteq \ma.$ Let  $x\notin \ma$. Then $x\in S\setminus \ma$ and by Lemma \ref{mms}, there exists a multiplicatively closed subset $X$ of $S$ such that $x\in X\subseteq S\setminus \ma.$ But $X\cap \ma=\emptyset$, and hence $x\notin \sqrt[z]{\ma},$ by Lemma \ref{rkt}.
\end{proof}

\begin{corollary}
The $z$-radical $\sqrt[z]{\ma}$ of an ideal $\ma$ of a semiring $S$ is the smallest $z$-semiprime ideal of $S$ containing $\ma$.
\end{corollary}

Strongly irreducible ideals, originally introduced as primitive ideals in \cite{F49} for commutative rings and referred to as quasi-prime ideals in \cite[p.,301, Exercise 34]{B72}, were first termed `strongly irreducible' in the context of noncommutative rings in \cite{B53}. In the context of semirings, a study of these ideals can be found in \cite{G99, I56}. In this section, we introduce the concepts of $z$-irreducible and $z$-strongly irreducible ideals in semirings and explore their relationships with $z$-prime and $z$-semiprime ideals.

\begin{definition}\label{isirr} Let $S$ be a semiring.
\begin{enumerate}
		
\item A $z$-ideal $\ma$ of $S$ is called \emph{$z$-irreducible} if  $\mb\cap \mb'= \ma$ implies  $\mb= \ma$ or $\mb'= \ma,$ for all $\mb$, $\mb'\in \z$.
		
\item A $z$-ideal $\ma$ of $S$ is called \emph{$z$-strongly irreducible} if  $\mb\cap \mb'\subseteq \ma$ implies  $\mb\subseteq \ma$ or $\mb'\subseteq \ma$ for all $\mb,$ $\mb'\in \z$.
\end{enumerate}
\end{definition}

It is clear from Definition \ref{isirr} that every $z$-strongly irreducible ideal is also $z$-irreducible, and by Lemma \ref{psi}, it is straightforward to see that every $z$-prime ideal is $z$-strongly irreducible. Our expectation is that the equivalence formulation (\ref{kxkx}) holds for both $z$-irreducible and $z$-strongly irreducible ideals, and the following result confirms this expectation.

\begin{proposition}\label{eqsi}
An ideal $\ma$ of a semiring $S$ is $z$-irreducible ($z$-strongly irreducible)  if and only if $\ma$ is irreducible (strongly irreducible) and a $z$-ideal of $S$.
\end{proposition}

\begin{proof}
We provide a proof for $z$-strongly irreducible ideals, that for $z$-irreducible ideals requiring only a
trivial change of terminology.
Suppose that $\ma$ is a $z$-strongly irreducible ideal and $\mb$, $\mb'$ are ideals of $S$ such that $\mb\cap \mb'\subseteq \ma.$ This implies
\[ \cz(\mb)\cap \cz(\mb')=\cz(\mb\cap \mb')\subseteq \cz(\ma)=\ma,\]
where, the first equality follows by Lemma \ref{lclk}(\ref{arbin}) and the inclusion by Lemma \ref{lclk}(\ref{ijcl}). By assumption, this implies that  $ \cz(\mb)\subseteq \ma$ or $ \cz(\mb')\subseteq \ma$. By Lemma \ref{lclk}(\ref{iclk}), we now have the claim. The proof of the converse statement is obvious.
\end{proof}

It is a well-established fact (see \cite[Proposition~7.33]{G99}) that a strongly irreducible ideal in a semiring possesses the following equivalent `elementwise' property: If $x$ and $y$ are elements of $S$ such that $\langle x \rangle \cap \langle y \rangle \subseteq \ma$, then $x\in \ma$ or $y\in \ma$. The next proposition shows that a similar result holds for $z$-strongly irreducible ideals.

\begin{proposition}
\label{abi} 
An ideal $\ma$ of a semiring $S$ is  $z$-strongly irreducible if and only if \,	 $\cz(\langle x\rangle)  \cap \cz(\langle y\rangle) \subseteq  \ma$ implies  $x\in \ma$ or $y\in \ma$, for all $x$, $y\in S$.
\end{proposition}

\begin{proof}
Let $\ma$ be a $z$-strongly irreducible ideal of $S$ and let $\langle x\rangle  \cap \langle y\rangle \subseteq  \ma$, for all $x,$ $y\in S .$ This implies
\[\cz(\langle x\rangle \cap \langle y\rangle)=\cz(\langle x\rangle)\cap  \cz(\langle y\rangle) \subseteq \cz (\ma)=\ma,\]
where the first equality follows from Lemma \ref{lclk}(\ref{arbin}).
Since $\ma$ is  $z$-strongly irreducible, we must have  $x\in \langle x\rangle\subseteq  \cz(\langle x\rangle) \subseteq \ma$ or  $y\in \langle y\rangle\subseteq  \cz(\langle y\rangle) \subseteq \ma$. To show the converse, suppose that $\ma$ is a $z$-ideal that is not $z$-strongly irreducible. Then there are $z$-ideals $\mb$ and $\mb'$ satisfy $\mb\cap \mb'\subseteq \ma,$ but $\mb\nsubseteq \ma$ and $\mb'\nsubseteq \ma.$ This implies there exist $x\in \mb$ and $y\in \mb'$ such that \[\cz(\langle x\rangle)  \cap \cz(\langle y\rangle) \subseteq \cz(\mb)\cap \cz(\mb')=\mb\cap \mb'\subseteq   \ma,\] but $x\notin \ma$ and $y\notin \ma$.
\end{proof}

In Theorem \ref{spkr}, we observed that a $z$-radical ideal can be expressed as the intersection of $z$-prime ideals containing it. Now, we will show that any proper $z$-ideal can be similarly represented in terms of $z$-irreducible ideals, and this extends \cite[Proposition 7.34]{G99}. However, before proceeding, let's establish a lemma.

\begin{lemma}\label{lir}
Suppose that $S$ is a semiring. Let $0\neq x\in S $ and $\ma$	be a proper $z$-ideal of  $S$ such that $x\notin \ma.$ Then there exists a $z$-irreducible ideal $\mb$ of $S$ such that $\ma\subseteq \mb$
and $x\notin \mb$.	
\end{lemma}

\begin{proof}
Let
$\{\mb_{i}\}_{i \in I}$ be a chain of $z$-ideals of $S$ such that $x\notin \mb_{i}\supseteq \ma$, for all $i \in I$. Then
\[x\notin \bigcup_{i \in I} \mb_{i}\supseteq \ma.
\]
By Zorn's lemma, there exists a maximal element $\mb$ of this chain. Suppose that  $\mb=\mb'\cap \mb''.$ By the maximality condition of $\mb$, we must have $x\in \mb'$ and $x\in \mb'',$ and hence $x\in \mb'\cap \mb''=\mb,$ a contradiction. Therefore, $\mb$ is the required $z$-irreducible ideal.
\end{proof}

\begin{proposition}\label{rpir}
If $\ma$ is a proper $z$-ideal of a semiring $S$, then $\ma=\bigcap_{\ma\subseteq \mb} J$, where $J$ is a $z$-irreducible ideal of $S$.
\end{proposition}

\begin{proof}
By Lemma \ref{lir}, there exists a $z$-irreducible ideal $\mb$ of  $S$ such that $\ma\subseteq \mb$. Suppose that  \[\mb'=\bigcap_{\ma\subseteq \mb} \left\{\mb\mid \mb\;\text{is $z$-irreducible}\right\}.\]
Then $\ma\subseteq \mb'$. We claim that $\mb'=\ma.$   If $\mb'\neq \ma$, then there exists an $x\in \mb'\setminus \ma$, and by Lemma \ref{lir}, there exists a $z$-irreducible ideal $\mb''$ such that $x\notin \mb''\supseteq \ma,$ a contradiction.
\end{proof}

As anticipated at the outset of this section, the following result demonstrates the connections between prime-type and irreducible-type $z$-ideals, and it partially extends \cite[Proposition 7.36]{G99}

\begin{proposition}\label{vpss}
A $z$-ideal of a bzi-semiring is $z$-prime if and only if it is $z$-semiprime and $z$-strongly irreducible.
\end{proposition} 

\begin{proof}
Let $\p$ be a $z$-prime ideal of a semiring $S$. Then by Proposition \ref{eqp}, $\p$ is a $z$-ideal and a prime ideal of  $S$. This implies that $\p$ is $z$-semiprime by Proposition \ref{eqpp}. From Lemma \ref{psi} and Proposition \ref{eqsi}, it follows that $\p$ is also $z$-strongly irreducible. Conversely, let $\p$ be a $z$-semiprime and $z$-strongly irreducible ideal. Suppose that $\ma$, $\mb\in \z $ satisfying $\ma\mb\subseteq \p$. Then 
\[(\ma\cap \mb)^2\subseteq \ma\mb\subseteq \p.\]
Since $\p$ is $z$-semiprime, this implies $\ma\cap \mb\subseteq \p$. But $\p$ is also $z$-strongly irreducible, and so $\ma\subseteq \p$ or $\mb\subseteq \p.$
\end{proof}

For commutative rings, it is a known result (see \cite[Theorem 2.1(ii)]{A08}) that every proper ideal is contained within a minimal strongly irreducible ideal. Interestingly, this property also extends to $z$-strongly irreducible ideals of semirings.

\begin{proposition}
Every $z$-proper ideal of a semiring is contained in a minimal $z$-strongly irreducible ideal.
\end{proposition}

\begin{proof}
Let $\ma$ be a proper $z$-ideal of a semiring $S$. Consider the chain $(\mathcal{E}, \subseteq)$,  where \[\mathcal{E}=\{J\mid \ma\subseteq \mb,\;\mb\;\text{is $z$-strongly irreducible}\}.\]
Since every maximal ideal of a semiring $S$ is strongly irreducible, by Proposition \ref{eqsm} and Proposition \ref{eqsi}, it follows that every $z$-maximal ideal is also $z$-strongly irreducible. Additionally, as shown in Lemma \ref{exma}, every proper $z$-ideal is contained in a $z$-maximal ideal. Therefore, the set $\mathcal{E}$ is nonempty. By applying Zorn's lemma, we conclude that $\mathcal{E}$ has a minimal element, which corresponds to our desired minimal $z$-strongly irreducible ideal.
\end{proof}

The following result demonstrates that when all $z$-ideals of a semiring are $z$-strongly irreducible, the proof of this fact is evident. This result generalizes \cite[Lemma 3.5]{A08}.

\begin{proposition}
Every $z$-ideal of a semiring $S$ is $z$-strongly irreducible if and only if  $\z $ is totally ordered. 
\end{proposition}

The next theorem partially characterizes arithmetical semirings, and it extends \cite[Theorem 7]{I56}.

\begin{theorem}\label{pcas}
In a arithmetical semiring $S$, a $z$-ideal of a semiring $S$ is
$z$-irreducible if and only if it is $z$-strongly irreducible. Conversely, 
if a $z$-irreducible ideal of a semiring $S$ is $z$-strongly
irreducible, then $S$ is arithmetical.
\end{theorem}

\begin{proof}
It has been shown in \cite[Theorem 3]{I56} that in an arithmetical semiring, irreducible and strongly irreducible ideals are equivalent. Thanks to Proposition \ref{eqsi}, we have then our claim. For the converse, \cite[Theorem 7]{I56} says that if irreducibility implies strongly irreducibility, then the semiring is arithmetical. Once again, applying Proposition \ref{eqsi} on this result, we get the converse.
\end{proof}

\begin{corollary}
In an arithmetical semiring, any $z$-ideal is the intersection of all $z$-strongly irreducible ideals containing it.
\end{corollary}

The next result shows when all ideals of an arithmetical semiring are $z$-ideals, and it extends \cite[Lemma 2.12]{MC19}.

\begin{theorem}
Let $S$ be an arithmetical semiring. Then every strongly irreducible ideal of $S$ is a $z$-ideal if and only if every ideal of $S$ is a $z$-ideal.
\end{theorem}

\begin{proof}
If every ideal of $S$ is a $z$-ideal, then in particular, every strongly irreducible ideal is also a $z$-ideal. For the converse, suppose that every strongly irreducible ideal is a $z$-ideal. Let $\ma$ be an ideal of $S$. Let $\mathcal{M}_x\supseteq \mathcal{M}_y$ and $y\in \ma$. If possible, assume that $x\notin \ma$. Consider the set poset \[\mathcal{E}:= \left\{\mb\in \id \mid \mb \supseteq \ma, x\notin \mb\right\}.\]
Since $\ma\in \mathcal{E}$, the set $\mathcal{E}$ is nonempty. If $\mathcal{K}$ is a chain in $\mathcal{E}$, then obviously $\bigcup \{\mathfrak{c}\mid \mathfrak{c}\in \mathcal{K}\}$ is an upper bound of $\mathcal{K}$ in $\mathcal{E}$. By Zorn's lemma , this implies that $\mathcal{E}$ has a maximal element $\p$ with the property: $\ma\subseteq \p$ and $x\notin \p$. We claim that $\p$ is a strongly irreducible ideal. Suppose that $\mathfrak{r}\cap \mathfrak{s}\subseteq \p$, for some $\mathfrak{r},$ $\mathfrak{s}\in \id.$ Suppose  $\mathfrak{r}\nsubseteq \p$ and $\mathfrak{s}\nsubseteq \p$. Since $\p$ is maximal with the property: $x\notin \p$, we obtain $x\in \p+ \mathfrak{r}$ and $x\in \p + \mathfrak{s}$. By applying arithmetical property of $S$, we have \[x\in (\p + \mathfrak{r})\cap (\p + \mathfrak{s})=\p + (\mathfrak{r} \cap \mathfrak{s})=\p,\]
a contradiction. Evidently, $y\in \p$ and $\mathcal{M}_x\supseteq \mathcal{M}_y$ implies that $x\in \p$, as $\p$ is a $z$-ideal. This is a contradiction. Hence $x\in \ma$, and hence $\ma$ is a $z$-ideal.
\end{proof}

We conclude with a couple of results on $z$-Noetherian semirings.

\begin{definition}\label{znsr}
A semiring $S$ is called \emph{$z$-Noetherian} if every ascending chain of $z$-ideals in $S$ eventually  stationary. 	
\end{definition}

The following proposition provides a generalization of \cite[Proposition 7.3]{N18}. It is important to note that this result also holds to irreducible ideals of Noetherian commutative rings

\begin{proposition}\label{nssi}
Let $S$ be a $z$-Noetherian semiring. Then every $z$-ideal of $S$ can be represented as an intersection of a finite number of $z$-irreducible ideals of $S$.
\end{proposition}

\begin{proof}
Suppose that 
\[\mathcal{F}=\left\{\mb\in  \z \mid \mb\neq \bigcap_{i=1}^n \q_i,\;\q_i \;\text{is $z$-irreducible}\right\}.\] To establish the claim, it is sufficient to prove that $\mathcal{F}=\emptyset$.
Since $S$ is $z$-Noetherian, $\mathcal{F}$ must have a maximal element, denoted by $\ma$. Since $\ma \in \mathcal{F}$, it is not a finite
intersection of $z$-irreducible ideals of $S$. This implies that $\ma$ is not $z$-irreducible. Hence, there are $z$-ideals $\mb$ and $\mb'$ such that  $\mb\supsetneq \ma$, $\mb'\supsetneq \ma$, and $\ma = \mb \cap \mb'.$ Since $\ma$ is
a maximal element of $\mathcal{F}$, we must have $\mb,$ $\mb' \notin \mathcal{F}.$ Therefore, $\mb$ and $\mb'$ are a finite intersection of
$z$-irreducible ideals of $S$. This, in turn, implies that $\ma$ is also a finite intersection of $z$-irreducible
ideals of $S$, a contradiction.
\end{proof}

In a $z$-Noetherian semiring we have another result, but for minimal $z$-prime ideals. This result extends \cite[Proposition 6.7]{L15}.

\begin{proposition}
\label{wnknp} 
If $S$ is a $z$-Noetherian semiring, then the set of minimal $z$-prime ideals of a semiring $S$ is finite.
\end{proposition}

\begin{proof}
We give a topological proof. If $S$ is $z$-Noetherian, then the topological space (endowed with Zariski topology) $\mathrm{Spec}_z(S) $ is also Noetherian, and thus $\mathrm{Spec}_z(S)$ has  finitely many irreducible
components. Now every irreducible closed subset of $\mathrm{Spec}_z(S) $
is of the form \[\mathcal{V}(\p)=\left\{\q\in \mathrm{Spec}_z(S) \mid \p\subseteq \q\right\},\] where $\p$ is a minimal $z$-prime ideal. Thus $\mathcal{V}(\p)$ is irreducible component if and only if $\p$ is a minimal $z$-prime ideal. Hence, the number of minimal $z$-prime ideals of $S$ is finite.
\end{proof}

\begin{remark}
As a continuation of this work, in \cite{Gos23}, we study further properties of $z$-ideals of several distinguished types of semirings. Specifically, we consider $z$-Noetherian semirings, factor semirings, and idempotent semirings. Additionally, we investigate the behaviour of various types of $z$-ideals (introduced here) under semiring homomorphisms. We also touch upon $z$-primary decompositions of semirings. It would be interesting to provide explicit nontrivial examples of
semirings $S$ for which $\mathrm{Spec}_z(S) $ is completely described. Finally, we explore hull-kernel topologies endowed on these distinguished types of $z$-ideals. In particular, It would be interesting to investigate  $\mathrm{Spec}_z(S) $ further as a topological space.
\end{remark}
\smallskip

\textbf{Acknowledgement.}\,
The author wishes to extend heartfelt gratitude to the anonymous referee for his/her thorough review and invaluable feedback, which greatly enhanced the paper's presentation.


\medskip

\begin{thebibliography}{100}

\bibitem{AK13} A. Altman and S. Kleiman, A term of commutative algebra, \textit{Worldwide center of mathematics, LLC}, 2013.

\bibitem{AAT13}
A. R. Aliabad, F. Azarpanah and A. Taherifar, Relative $z$-ideals in commutative rings, \textit{Comm. Algebra}, \textbf{41}(1) (2013), 325--341.

\bibitem{AM13}
\bysame and R. Mohamadian,  On $z$-ideals and $z^{\circ}$-ideals of power series rings. \textit{J. Math. Ext.}, \textbf{7}(2) (2013), 93--108.

\bibitem{ABN20}
\bysame,  M. Badie, and S. Nazari, An extension of  $z$-ideals and  $z^{\circ}$-ideals,
\textit{Hacet. J. Math. Stat.}, \textbf{49}(1) (2020), 254--272.

\bibitem{AKRA99}
F. Azarpanah,  O.A.S. Karamzadeh, and A. Rezai Aliabad, On $z^{\circ}$-ideals in $C(X)$, \textit{Fund. Math.}, \textbf{160} (1999), 15--25.

\bibitem{AM07}
\bysame and R. Mohamadian, $\sqrt{z}$-ideals and $\sqrt{z^{\circ}}$-ideals in $C(X)$, \textit{Acta Mathematica Sinica.}, \textbf{23} (2007), 989--
996.

\bibitem{AP20}
\bysame and M. Parsinia, On the sum of  $z$-ideals in subrings of  $C(X)$,
\textit{J. Commut. Algebra}, \textbf{12}(4) (2020), 459--466.

\bibitem{A08} A. Azizi, Strongly irreducible ideals, \emph{J. Aust. Math. Soc.}, \textbf{84} (2008), 145--154.

\bibitem{B53} R. L.
Blair, Ideal lattices and the structure of rings, \emph{Trans. the Amer. Math. Soc.}, \textbf{75} (1953),  136--153.   
\bibitem{B72} N. Bourbaki, \emph{Elements of mathematics:
Commutative
Algebra}, Addison-Wesley, Reading, MA, 1972.

\bibitem{Ben21}
A. Benhissi, Chain condition on  $z$-ideals, \textit{Ric. Mat.}, \textbf{70}(2) (2021), 347--352.

\bibitem{BM20}
\bysame and A. Maatallah, A question about higher order  $z$-ideals in commutative rings,
Quaest. Math., \textbf{43}(8) (2020),  1155--1157.

\bibitem{CC19}
A. Connes and C. Consani, Homological algebra in characteristic one, \textit{Higher Struct. J.}, \textbf{3}(1)
(2019), 155--247.

\bibitem{Dub16}
T. Dube, A note on lattices of $z$-ideals of $f$-rings, \textit{N.Y. J. Math.}, \textbf{22} (2016), 351--361.

\bibitem{Dub18}
\bysame,  Some connections between frames of radical ideals and frames of $z$-ideals, \textit{Algebra Univers.}, \textbf{79}(7) (2018), 18 pages.

\bibitem{DO16}
\bysame and O. Ighedo, Higher order  $z$-ideals in commutative rings,
\textit{Miskolc Math. Notes}, \textbf{17}(1) (2016), 171--185.

\bibitem{Mar72}
G. De Marco, On the countably generated  $z$-ideals of  $C(X)$, {Proc. Amer. Math. Soc.}, \textbf{31} (1972), 574--576.

\bibitem{Pag81}
B. de Pagter,  On $z$-ideals and $d$-ideals in Riesz spaces III, \textit{Nederl. Akad. Wetensch. Indag. Math.}, \textbf{43}(4) (1981), 409--422.

\bibitem{F49} L. Fuchs, \"{U}ber die Ideale arithmetischer Ringe, \emph{Comment. Math. Helv.}, \textbf{23} (1949), 334--341.

\bibitem{GG16}
J. Giansiracusa and N. Giansiracusa, Equations of tropical varieties, \textit{Duke Math. J.}, \textbf{165}(18)
(2016), 3379--3433.

\bibitem{GJ60} 
L. Gillman and M. Jerison, \emph{Rings of continuous functions}, D. Van Nostrand Company, Inc., 1960. 

\bibitem{G99} J. S. Golan, \textit{Semirings and their applications}, Springer, 1999. 

\bibitem{Gos23}
A. Goswami, On $z$-ideals and $z$-closure operations of semirings, II (in preparation).

\bibitem{H58}
M. Henriksen, Ideals in semirings with commutative addition, \textit{Amer. Math. Soc.
Notices}, \textbf{6}(3) 31 (1958), 321.

\bibitem{HMW03}
\bysame, J. Mart\'{i}nez, and R. G. Woods, Spaces  $X$  in which all prime  $z$-ideals of  $C(X)$  are minimal or maximal,
\textit{Comment. Math. Univ. Carolin.}, \textbf{44}(2) (2003), 261--294.

\bibitem{HP80}
C. B. Huijsmans and B. de Pagter,  On $z$-ideals and $d$-ideals in Riesz spaces II, \textit{Nederl. Akad. Wetensch. Indag. Math.}, \textbf{42}(4) (1980), 391--408. 

\bibitem{HP80-1}
\bysame,  On $z$-ideals and $d$-ideals in Riesz spaces I, \textit{Nederl. Akad. Wetensch. Indag. Math.}, \textbf{42}(2) (1980), 183--195. 

\bibitem{I59} K. Ilzuka, On the Jacobson radical of a semiring, \textit{Tohoku Math. J.}, \textbf{11} (1959), 409--421. 

\bibitem{I56} K. Is\'{e}ki, Ideal theory of semiring , \textit{Proc. Japan. Acad.}, \textbf{32} (1956), 554--559.

\bibitem{JK19}
V. Joshi and S. Kavishwar, $z$-ideals in lattices,
Acta Sci. Math. (Szeged), \textbf{85}(1-2) (2019), 59--68.


\bibitem{JRT22}
J. June, S. Ray, and J. Tolliver, Lattices, spectral spaces, and closure operations on
idempotent semirings, \textit{J. Algebra}, \textbf{594} (2022), 313--363.

\bibitem{Koh57}
C. W. Kohls, Ideals in rings of continuous functions, \textit{Fund. Math.}, \textbf{45} (1957), 28--50.

\bibitem{L01}
T. Y. Lam, \textit{A first course in noncommutative rings}, Springer, 2001.

\bibitem{L15}
P. Lescot, Prime and primary ideals in semirings, \textit{Osaka J. Math.}, \textbf{52}(3) (2015), 721--737.

\bibitem{Lor19}
O. Lorscheid, Tropical geometry over the tropical hyperfield, preprint, arXiv:1907.01037, 2019.

\bibitem{MB22}
A. Maatallah and A. Benhissi, Higher order  $z$-ideals of special rings,
\textit{Beitr. Algebra Geom.}, \textbf{63}(1) (2022), 167--177.

\bibitem{Don80}
A. Le Donne, On a question concerning countably generated  $z$-ideals of  $C(X)$, 
\textit{Proc. Amer. Math. Soc.}, \textbf{80}(3) (1980),  505--510.

\bibitem{Man68}
M. Mandelker,  Prime  $z$-ideal structure of  $C(R)$,
\textit{Fund. Math.}, \textbf{63} (1968), 145--166.

\bibitem{Mas73}
G. Mason,  $z$-ideals and prime ideals, \textit{J. Algebra}, \textbf{26} (1973), 280--297.

\bibitem{Mas80}
\bysame,  Prime  $z$-ideals of  $C(X)$  and related rings,
\textit{Canad. Math. Bull.}, \textbf{23}(4) (1980), 437--443.

\bibitem{Moh14}
R. Mohammadian, Positive semirings, \textit{J. Adv. Math. Modelling}, \textbf{3}(2) (2014), 103--125.

\bibitem{MC19}
C. S. Manjarekar and A. N. Chavan, $Z$-elements and  $z_j$-elements in multiplicative lattices,
\textit{Palest. J. Math.}, \textbf{8}(1) (2019), 138--147.

\bibitem{MJ20}
M. Masoudi-Arani and R. Jahani-Nezhad,  Generalization of  $z$-ideals in right duo rings, Hacet. \textit{J. Math. Stat.}, \textbf{49}(4) (2020), 1423--1436.

\bibitem{MB20}
A. Maatallah and A. Benhissi, A note on  $z$-ideals and  $z^{\circ}$-ideals of the formal power series rings and polynomial rings in an infinite set of indeterminates, \textit{Algebra Colloq.}, \textbf{27}(3) (2020), 495--508.

\bibitem{N18} P. Nasehpour, Some remarks on ideals of commutative semirings, \textit{Quasigroups Related Systems}, \textbf{26}
(2018), 281--298.

\bibitem{SA92}
M. K. Sen and M. R. Adhikari, On $k$-ideals of semirings, \textit{Internat. J. Math. \& Math. Sci.}, \textbf{15}(2) (1992), 347--350.

\bibitem{SA93}
\bysame, On maximal $k$-ideals of semirings, \textit{Proc. Amer. Math.
Soc.} \textbf{118} (1993), 699--703.


\bibitem{SZ55} W. Slowikowski and W. Zawadowski, A generalisation of maximal ideals method
of Stone and Gelfand, \textit{Fund. Math.}, \textbf{42} (1955), 216--231.

\bibitem{Vir10}
O. Viro, Hyperfields for tropical geometry I. hyperfields and dequantization, preprint, arXiv:
1006.3034, 2010.

\bibitem{VMS19}
E. M. Vechtomov, A. V. Mikhalev, and V. V. Sidorov, Semirings of continuous functions, \textit{J. Math. Sci.}, \textbf{237}(2) (2019), 191--244.

\bibitem{XZ09}
Y. P. Xiao and T. R. Zou, $Z$-implicative algebra and fuzzy $Z$-ideal,
\textit{J. Fuzzy Math.}, \textbf{17}(2) (2009), 335--350.
\end{thebibliography}
\end{document}